\documentclass[12pt]{article}
\usepackage{amsmath}
\usepackage{amsthm, amscd, amssymb, amsfonts, amsbsy}
\usepackage[usenames, dvipsnames]{color}
\usepackage{enumerate}
\usepackage[colorlinks=true,linkcolor=blue,citecolor=red]{hyperref}
\usepackage{mathrsfs}
\usepackage{verbatim}
\usepackage{marginnote}
\usepackage{etoolbox}
\usepackage{cite}
\patchcmd{\abstract}{\scshape\abstractname}{\textbf{\abstractname}}{}{}

\numberwithin{equation}{section}

\theoremstyle{plain}
\newtheorem{theorem}{Theorem}[section]
\newtheorem{lemma}[theorem]{Lemma}

\newtheorem{corollary}[theorem]{Corollary}

\theoremstyle{definition}

\theoremstyle{remark}
\newtheorem{remark}[theorem]{Remark}

\newcommand{\aint}{{\int\hspace*{-4.3mm}\diagup}}
\makeatletter
\def\dashint{\operatorname%
{\,\,\text{\bf--}\kern-.98em\DOTSI\intop\ilimits@\!\!}}
\makeatother

\def\bR{\mathbb{R}}

\def\bZ{\mathbb{Z}}

\newlength{\defbaselineskip}
\begin{document}

\title{Gradient estimates of solutions to the insulated conductivity problem in dimension greater than two}
\author{YanYan Li\footnote{Department of
Mathematics, Rutgers University, 110 Frelinghuysen Rd, Piscataway,
NJ 08854, USA. Email: yyli@math.rutgers.edu.}~\footnote{Partially supported by NSF Grants DMS-1501004, DMS-2000261, and Simons Fellows Award 677077.}\quad and \quad Zhuolun Yang\footnote{Department of Mathematics, Rutgers University, 110 Frelinghuysen Rd, Piscataway,
NJ 08854, USA. Email: zy110@math.rutgers.edu.}~\footnote{Partially supported by NSF Grants DMS-1501004 and DMS-2000261.}}
\date{}

\maketitle

\begin{abstract}
We study the insulated conductivity problem with inclusions embedded in a bounded domain in $\bR^n$. The gradient of solutions may blow up as $\varepsilon$, the distance between inclusions, approaches to $0$. An upper bound for the blow up rate was proved to be of order $\varepsilon^{-1/2}$. The upper bound was known to be sharp in dimension $n = 2$. However, whether this upper bound is sharp in dimension $n \ge 3$ has remained open. In this paper, we improve the upper bound in dimension $n \ge 3$ to be of order $\varepsilon^{-1/2 + \beta}$, for some $\beta > 0$.
\end{abstract}

\section{Introduction and main result}

Let $\Omega$ be a bounded domain in $\bR^n$ with $C^{2}$ boundary, and let $D_{1}^{*}$ and $D_{2}^{*}$ be two open sets whose closure belongs to $\Omega$, touching only at the origin with the inner normal vector of $\partial{D}_{1}^{*}$ pointing in the positive $x_{n}$-direction. Translating $D_{1}^{*}$ and $D_{2}^{*}$ by $\frac{\varepsilon}{2}$ along $x_{n}$-axis, we obtain
$$D_{1}^{\varepsilon}:=D_{1}^{*}+(0',\frac{\varepsilon}{2}),\quad\mbox{and}\quad\,D_{2}^{\varepsilon}:=D_{2}^{*}-(0',\frac{\varepsilon}{2}).$$
When there is no confusion, we drop the superscripts $\varepsilon$ and denote $D_{1}:=D_{1}^{\varepsilon}$ and $D_{2}:=D_{2}^{\varepsilon}$. Denote $\widetilde{\Omega} := \Omega \setminus \overline{(D_1 \cup D_2)}$, we consider the following elliptic equation with Dirichlet boundary data:
\begin{equation}\label{equk}
\begin{cases}
\mathrm{div}\Big(a_{k}(x)\nabla{u}_{k}\Big)=0&\mbox{in}~\Omega,\\
u_{k}=\varphi(x)&\mbox{on}~\partial\Omega,
\end{cases}
\end{equation}
where $\varphi\in{C}^{2}(\partial\Omega)$ is given, and
$$a_{k}(x)=
\begin{cases}
k\in(0,\infty)&\mbox{in}~D_{1}\cup{D}_{2},\\
1&\mbox{in}~\widetilde{\Omega}.
\end{cases}
$$
The equation above can be considered as a simple model for electric conduction, where $a_k$ refers to conductivities, which can be assumed to be 1 in the matrix after normalization, and the solution $u_k$ gives the voltage potential. From an engineering point of view, it is very important to estimate $\nabla u_k$, which represents the electric fields, in the narrow region between the inclusions. This problem is analogous to a linear system of elasticity studied by Babu\v{s}ka, Andersson, Smith and Levin \cite{BASL}, where they analyzed numerically that, when the ellitpicity constants are bounded away from $0$ and infinity, gradient of solutions remain bounded independent of $\varepsilon$, the distance between inclusions. Bonnetier and Vogelius \cite{BV} proved that for a fixed $k$, $|\nabla u_k|$ remains bounded as $\varepsilon$ goes to 0, for circular inclusions $D_1$ and $D_2$ in dimension $n = 2$. This result was extended by Li and Vogelius \cite{LV} to general second order elliptic equation of divergence form with piecewise H\"older coefficients and general shape of inclusions $D_1$ and $D_2$ in any dimension. Furthermore, they established a stronger $C^{1,\alpha}$ control of $u_k$, which is independent of $\varepsilon$, in each region. Li and Nirenberg \cite{LN} further extended this $C^{1,\alpha}$ result to general second order elliptic systems of divergence form.

When $k$ equals to $\infty$ (perfect conductor) or $0$ (insulator), it was shown in \cite{Kel, BudCar, Mar} that the gradient of solutions generally becomes unbounded, as $\varepsilon \to 0$. When $k$ goes to $\infty$, $u_k$ converges to the solution of the following perfect conductivity problem: 
\begin{equation}\label{equinfty}
\begin{cases}
\Delta{u}=0&\mbox{in}~\widetilde{\Omega},\\
u=C_i \mbox{ (Constants)}&\mbox{on}~\partial{D}_{i},~i=1,2,\\
\int_{\partial{D}_{i}}\frac{\partial{u}}{\partial\nu}=0&i=1,2,\\
u=\varphi(x)&\mbox{on}~\partial\Omega.
\end{cases}
\end{equation}
When $k$ goes to $0$, $u_k$ converges to the solution of the following insulated conductivity problem:
\begin{equation}\label{equzero}
\left\{
\begin{aligned}
-\Delta u &=0 \quad \mbox{in }\widetilde{\Omega},\\
\frac{\partial u}{\partial \nu} &= 0 \quad \mbox{on}~\partial{D}_{i},~i=1,2,\\
 u &= \varphi \quad \mbox{on } \partial \Omega.
\end{aligned}
\right.
\end{equation}
See, e.g., Appendix of \cite{BLY1} and \cite{BLY2} for derivations of the above equations. Here $\nu$ denotes the inward unit normal vectors on $\partial D_i$, $i = 1,2$.

Ammari et al. proved in \cite{AKLLL} and \cite{AKL}, among other things, the following. Let $D_1^*$ and $D_2^*$ be unit balls in $\bR^2$, and let $H$ be a harmonic function in $\bR^2$. They considered the perfect and insulated conductivity problems in $\bR^2$:
\begin{equation*}
\begin{cases}
\Delta{u}=0&\mbox{in}~\bR^2\setminus\overline{(D_1 \cup D_2)},\\
u=C_i \mbox{ (Constants)}&\mbox{on}~\partial{D}_{i},~i=1,2,\\
\int_{\partial{D}_{i}}\frac{\partial{u}}{\partial\nu}=0&i=1,2,\\
u(x)-H(x) = O(|x|^{-1})&\mbox{as}~|x| \to \infty,
\end{cases}
\end{equation*}
and
\begin{equation*}
\begin{cases}
\Delta{u}=0&\mbox{in}~\bR^2\setminus\overline{(D_1 \cup D_2)},\\
\frac{\partial u}{\partial \nu} = 0 &\mbox{on}~\partial{D}_{i},~i=1,2,\\
u(x)-H(x) = O(|x|^{-1})&\mbox{as}~|x| \to \infty.
\end{cases}
\end{equation*}
In both cases, they proved that for some $C$ independent of $\varepsilon$,
$$\| \nabla u\|_{L^\infty(B_4)} \le C \varepsilon^{-1/2}.$$
They also showed that the upper bounds are optimal in the sense that for appropriate $H$,
$$\| \nabla u\|_{L^\infty(B_4)} \ge \varepsilon^{-1/2}/C.$$
Yun extended in \cite{Y1} and \cite{Y2} the results allowing $D_1^*$ and $D_2^*$ to be any bounded strictly convex smooth domains.

The above gradient estimates were localized and extended to higher dimensions by Bao, Li and Yin in \cite{BLY1} and \cite{BLY2}. For the perfect conductor case, they considered problem \eqref{equinfty} and proved in \cite{BLY1} that
\begin{equation*}
\begin{cases}
\| \nabla u \|_{L^\infty(\widetilde{\Omega})} \le C\varepsilon^{-1/2} \|\varphi\|_{C^2(\partial \Omega)} &\mbox{when}~n=2,\\
\| \nabla u \|_{L^\infty(\widetilde{\Omega})} \le C|\varepsilon \ln \varepsilon|^{-1} \|\varphi\|_{C^2(\partial \Omega)} &\mbox{when}~n=3,\\
\| \nabla u \|_{L^\infty(\widetilde{\Omega})} \le C\varepsilon^{-1} \|\varphi\|_{C^2(\partial \Omega)} &\mbox{when}~n\ge 4.
\end{cases}
\end{equation*}
The above bounds were shown to be optimal in the paper. For further works on the perfect conductivity problem and closely related ones, see e.g. \cite{ACKLY,BT1,BT2,DL,KLY1,KLY2,L,LLY,LWX,BLL,BLL2,DZ,KL,CY,ADY,Gor,LimYun} and the references therein.

For the insulated problem \eqref{equzero}, it was proved in \cite{BLY2} that
\begin{equation}\label{insulated_upper_bound}
\| \nabla u \|_{L^\infty(\widetilde{\Omega})} \le C\varepsilon^{-1/2} \|\varphi\|_{C^2(\partial \Omega)}\quad \mbox{when}~n\ge 2.
\end{equation}
The upper bound is optimal for $n = 2$ as mentioned earlier, while it was not known whether it is optimal in dimensions $n \ge 3$.

Yun \cite{Y3} considered the following insulated problem in $\bR^3$ minus 2 balls: Let $H$ be a harmonic function in $\bR^3$, $D_1 = B_1\left(0,0,1+\frac{\varepsilon}{2} \right)$, and $D_2 = B_1\left(0,0,-1-\frac{\varepsilon}{2} \right)$,
\begin{equation*}
\begin{cases}
\Delta{u}=0&\mbox{in}~\bR^3\setminus\overline{(D_1 \cup D_2)},\\
\frac{\partial u}{\partial \nu} = 0 &\mbox{on}~\partial{D}_{i},~i=1,2,\\
u(x)-H(x) = O(|x|^{-2})&\mbox{as}~|x| \to \infty.
\end{cases}
\end{equation*}
He proved that for some positive constant $C$ independent of $\varepsilon$,
$$\max_{|x_3|\le \varepsilon/2}|\nabla u(0,0,x_3)| \le C \varepsilon^{\frac{\sqrt{2}-2}{2}}.$$
He also showed that this upper bound of $|\nabla u|$ on the $\varepsilon$-segment connecting $D_1$ and $D_2$ is optimal for $H(x) \equiv x_1$.

In this paper, we focus on the insulated conductivity problem \eqref{equzero} in dimension $n \ge 3$, and improve the upper bound \eqref{insulated_upper_bound} to the rate $\varepsilon^{-1/2 + \beta}$, for some $\beta > 0$. Analogous questions for elliptic system are still open, and we give some discussions in Section 4. We point out that the insulator case for Lam\'{e} systems in dimension $n = 2$ was studied by Lim and Yu \cite{LimYu}.

From now on, we assume that $\partial{D}_{1}^{*}$ and $\partial{D}_{2}^{*}$ are $C^2$, and they are relatively convex near the origin. That is, for some positive constants $R_0, \kappa$, we assume that when $0<|x'|<R_{0}$, $\partial{D}_{1}^{*}$ and $\partial{D}_{2}^{*}$ are respectively the graphs of two $C^{2}$ functions $f$ and $g$ in terms of $x'$, and 
$$f(x')>g(x'),\quad\mbox{for}~~0<|x'|<R_{0},$$
\begin{equation}\label{fg_0}
f(0')=g(0')=0,\quad\nabla_{x'}f(0')=\nabla_{x'}g(0')=0,
\end{equation}
\begin{equation}\label{fg_1}
\nabla^{2}_{x'}(f-g)(x')\geq\kappa I_{n-1},\quad\mbox{for}~~0<|x'|<R_{0},
\end{equation}
where $I_{n-1}$ denotes the $(n-1) \times (n-1)$ identity matrix.
Let $a(x) \in C^\alpha(\overline{\widetilde{\Omega}})$, for some $\alpha \in (0,1)$, be a symmetric, positive definite matrix function satisfying
$$\lambda \le a(x) \le \Lambda, \quad \mbox{for }x \in \widetilde{\Omega},$$
for some positive constants $\lambda, \Lambda$. Let $\nu = (\nu_1, \cdots, \nu_n)$ denote the unit normal vector on $\partial D_1$ and $\partial D_2$, pointing towards the interior of $D_1$ and $D_2$.
We consider the following insulated conductivity problem:
\begin{equation}\label{main_problem}
\left\{
\begin{aligned}
-\partial_i (a^{ij} \partial_j u) &=0 \quad \mbox{in }\widetilde{\Omega},\\
a^{ij} \partial_j u \nu_i &= 0 \quad \mbox{on } \partial (D_1 \cup D_2),\\
 u &= \varphi \quad \mbox{on } \partial \Omega,
\end{aligned}
\right.
\end{equation}
where $\varphi \in C^{2}(\partial \Omega)$ is given. \\

For $0 < \,r\leq\,R_{0}$, we denote
\begin{align}\label{domain_def_Omega}
\Omega_{x_0,r}:=\left\{(x',x_{n})\in \widetilde{\Omega}~\big|~-\frac{\varepsilon}{2}\right.&\left.+g(x')<x_{n}<\frac{\varepsilon}{2}+f(x'),~|x' - x_0'|<r\right\},\nonumber\\
\Gamma_+ :=& \left\{ x_n = \frac{\varepsilon}{2}+f(x'),~|x'|<R_0\right\},\\
\Gamma_- :=& \left\{ x_n = -\frac{\varepsilon}{2}+g(x'),~|x'|<R_0\right\}. \nonumber
\end{align}
Since the blow-up of gradient can only occur in the narrow region between $D_1$ and $D_2$, we will focus on the following problem near the origin:

\begin{equation}\label{main_problem_narrow}
\left\{
\begin{aligned}
-\partial_i (a^{ij} \partial_j u) &=0 \quad \mbox{in }\Omega_{0,R_0},\\
a^{ij} \partial_j u \nu_i &= 0 \quad \mbox{on } \Gamma_+ \cup \Gamma_-,\\
\end{aligned}
\right.
\end{equation}
where $\nu = (\nu_1, \cdots, \nu_n)$ denotes the unit normal vector on $\Gamma_+$ and $\Gamma_-$, pointing upward and downward respectively.

Here is our main result in the paper.\\

\begin{theorem}\label{main_thm}
Let $f,g,a,\alpha$ be as above, and let $u \in H^1(\Omega_{0,R_0})$ be a solution of \eqref{main_problem_narrow} in dimension $n \ge 3$. There exist positive constants $r_0, \beta$ and $C$ depending only on $n$, $\lambda$, $\Lambda$, $R_0$, $\kappa$, $\alpha$, $\|a\|_{C^\alpha(\Omega_{0,R_0})}$, $\|f\|_{C^{2}(\{|x'| \le R_0\})}$ and $\|g\|_{C^{2}(\{|x'| \le R_0\})},$ such that
\begin{equation}\label{main_result}
|\nabla u (x_0)| \le C  \| u\|_{L^\infty(\Omega_{0,R_0})} \left(\varepsilon + |x_0'|^2 \right) ^{-1/2 + \beta},
\end{equation}
for all $x_0 \in \Omega_{0 , r_0}$ and $\varepsilon \in (0,1)$.\\
\end{theorem} 

Let $u \in H^1(\widetilde{\Omega})$ be a weak solution of \eqref{main_problem}. By the maximum principle and the gradient estimates of solutions of elliptic equations,
\begin{equation}\label{boundedness_u}
\|u\|_{L^\infty(\widetilde{\Omega})} \le \|\varphi\|_{L^\infty(\partial \Omega)},
\end{equation}
and
$$\| \nabla u\|_{L^\infty(\widetilde{\Omega} \setminus \Omega_{0, r_0} )} \le C\| \varphi\|_{C^{2}(\partial \Omega)}.$$
Therefore, a corollary of Theorem \ref{main_thm} is as follows.\\

\begin{corollary}
Let $u \in H^1(\widetilde{\Omega})$ be a weak solution of \eqref{main_problem} in dimension $n \ge 3$. There exist positive constants $\beta$ and $C$ depending only on $n$, $\lambda$, $\Lambda$, $R_0$, $\kappa$, $\|a\|_{C^\alpha}$, $\|f\|_{C^{2}}$ and $\|g\|_{C^{2}},$ such that
\begin{equation}\label{main_result_2}
\|\nabla u\|_{L^\infty(\widetilde{\Omega})} \le C  \| \varphi\|_{C^{2}(\partial \Omega)} \varepsilon ^{-\frac{1}{2} + \beta}.
\end{equation}\\
\end{corollary}

\begin{remark}
If there are more than two inclusions, estimate \eqref{main_result_2} still holds, with $\varepsilon$ being the minimal distance between inclusions.
\end{remark}

The rest of this paper will be organized as follows. In section 2, we prove a lemma which is used in the proof of Theorem \ref{main_thm}. Theorem \ref{main_thm} is proved in Section 3. In section 4, we give a gradient estimate to a problem for elliptic systems analogous to problem \eqref{main_problem_narrow}.\\

\section{A regularity lemma}

In this section, we give a regularity lemma for elliptic systems (elliptic equations when $N = 1$). Let us first describe the nature of domains and operators. We define $S$ to be a cylinder
$$S = \{ (x',x_n) \in \bR^n ~\big|~ |x'| < 1, |x_n| < 1 \},$$
and some constants $c_m$, with $0 \le m \le l$, such that
$$-1 = c_0 < c_1 < \cdots < c_l = 1,$$
and denote the integer $m_0$ to be the integer such that
$$c_{m_0 - 1} \le 0 < c_{m_0}.$$
We divide the domain $S$ into $l$ parts by setting
$$\Omega_m = \{x \in S ~\big|~ c_{m-1} < x_n < c_m \}, \quad \mbox{for }1 \le m \le l.$$
For $1 \le \alpha, \beta \le n, 1 \le i,j \le N$, let $A^{\alpha \beta}_{ij}(x)$ be a function such that
$$\| A^{\alpha \beta}_{ij} \|_{L^\infty(S)} \le \Lambda,$$
$$\int_S A^{\alpha \beta}_{ij}(x) \partial_\alpha \varphi_i(x) \partial_\beta \varphi_j(x) \ge \lambda \int_S |\nabla \varphi|^2, \quad \forall \varphi \in H_0^1(S; \bR^N),$$
for some $\lambda, \Lambda > 0$, 
and for each $1 \le m \le l$, $A^{\alpha \beta}_{ij}(x) \in C^\mu(\overline{\Omega}_m)$, for some $0 < \mu <1$. We denote $(A^{\alpha \beta}_{ij}(x))$ by $A(x)$.

For $1 \le \alpha \le n, 1 \le i \le N$, let
\begin{align*}
H(x) &= \{H_i\} \in L^\infty(S),\\
G(x) &= \{G_i^\alpha\} \in C^\mu(\overline{\Omega}_m),
\end{align*}
for all $m = 1 ,\cdots , l$. Then we have the following interior gradient estimate.\\

\begin{lemma}\label{gradient_lemma}
Let $A(x)$, $H(x)$ and $G(x)$ be as above. There exists a positive constant $C$, depending only on $n, \mu, \lambda, \Lambda$ and an upper bound of $\{\|A\|_{C^\mu(\overline{\Omega}_m)}\}_{m = 1}^l$, such that if $u \in H^1(S; \bR^N)$ is a weak solution to
$$\partial_\alpha (A^{\alpha \beta}_{ij}(x) \partial_\beta u_j) = H_i + \partial_\alpha G^\alpha_i \quad \mbox{in }S,$$
then
$$\| u\|_{L^\infty(\frac{1}{2}S)} + \| \nabla u\|_{L^\infty(\frac{1}{2}S)} \le C\left(\|u\|_{L^2(S)} + \|H\|_{L^\infty(S)} + \max_{1 \le m \le l} \|G\|_{C^\mu(\overline{\Omega}_m)} \right).$$\\
\end{lemma}

\begin{remark}
We point out that the constant $C$ in the Lemma is independent of $l$.
\end{remark}

\begin{proof}
The proof of Lemma \ref{gradient_lemma} is a modification of the proof of Proposition 4.1 in \cite{LN}. Even though the constant $C$ in \cite[Proposition 4.1]{LN} depends on $l$, the number of subdomains we divide in the domain $S$, this dependence only enters in estimating the quantities $\| A - \bar{A} \|_{Y^{1+\alpha, 2}}$, $\| G - \bar{G} \|_{Y^{1+\alpha, 2}}$, and $\| H - \bar{H} \|_{Y^{\alpha, 2}}$ which will be defined below. We will show that such quantity is independent of $l$ due to the nature of our domain $S$, and hence the constant $C$ in Lemma \ref{gradient_lemma} is independent of $l$.\\
\end{proof}

For $s > 0, 1 < p < \infty$, we define the norm
$$\|f\|_{Y^{s,p}}:= \sup_{0 < r \le 1} r^{1-s} \left( \aint_{rS} |f|^p \right)^{1/p}.$$
We define a piecewise-constant coefficients $\bar{A}$ associated to $A$ by setting
$$\bar{A}(x) := \left\{\begin{aligned}
\lim_{x \in \Omega_m, x \to (0', c_{m-1})} &A(x), &&\mbox{if }x \in \Omega_m, m > m_0;\\
&A(0), &&\mbox{if }x \in \Omega_{m_0};\\
\lim_{x \in \Omega_m, x \to (0', c_m)} &A(x), &&\mbox{if }x \in \Omega_m, m < m_0.
\end{aligned}
\right.$$
Similarly, we can define piecewise-constant tensor $\bar{G}$ associated to $G$. We also define a constant vector $\bar{H}$ associated to $H$ by
$$\bar{H} := \aint_S H.$$\\

\begin{lemma}
Let $A, \bar{A}, H, \bar{H}, G, \bar{G}$ be as above. Then there exists a positive constant $C$, depending only on $n$, such that
\begin{align*}
\| A - \bar{A} \|_{Y^{1+\mu , 2}} & \le C\max_{1 \le m \le l} \|A\|_{C^\alpha(\overline{S}_m)},\\
\| G - \bar{G} \|_{Y^{1+\mu , 2}} & \le C\max_{1 \le m \le l} \|G\|_{C^\alpha(\overline{S}_m)},\\
\| H - \bar{H} \|_{Y^{\mu , 2}} & \le  C\|H\|_{L^\infty(S)}.\\
\end{align*}
\end{lemma}
\begin{proof}
The last inequality follows immediately from the definition of $Y^{\mu,2}$ and $\bar{H}$:
\begin{align*}
\| H - \bar{H} \|_{Y^{\mu , 2}} \le \sup_{0 < r \le 1} r^{1-\mu} \left( \aint_{rS} |H - \bar{H}|^2 \right)^{1/2} \le C\|H\|_{L^\infty(S)}.
\end{align*}
By a direct computation, we will have
\begin{align*}
\left( \aint_{rS} |A - \bar{A}|^2 \right)^{1/2} &\le \left( \frac{1}{|rS|} \sum_{m = 1}^l \int_{rS \cap S_m} |A(x) - \bar{A}(x)|^2 \, dx \right)^{1/2}\\
&\le \left[ \frac{1}{|rS|} \left( \sum_{m = 1}^{m_0 - 1} \|A\|^2_{C^\mu(\overline{S}_m)} \int_{rS \cap S_m} |x - (0', c_m)|^{2\mu} \, dx \right.\right.\\
&+ \|A\|^2_{C^\mu(\overline{S}_{m_0})} \int_{rS \cap S_{m_0}} |x|^{2\mu} \, dx\\
+ \sum_{m = m_0 + 1}^{l} &\left. \left.\|A\|^2_{C^\mu(\overline{S}_{m-1})} \int_{rS \cap S_{m-1}} |x - (0', c_{m-1})|^{2\mu} \,dx \right)  \right]^{1/2} \\
&\le \max_{1 \le m \le l} \|A\|_{C^\mu(\overline{S}_m)} \left( \aint_{rS} |x|^{2\mu} \, dx \right)^{1/2}\\
&\le C\max_{1 \le m \le l} \|A\|_{C^\mu(\overline{S}_m)}r^\mu.
\end{align*}
This proves the first inequality. The second inequality follows similarly.

\end{proof}

\section{Proof of Theorem \ref{main_thm}}

In this section, we prove Theorem \ref{main_thm}. For a small $r_0$ independent of $\varepsilon$, and any $x_0 \in \Omega_{0, r_0}$, we estimate $|\nabla u(x_0)|$ as follows: First we establish a Harnack inequality in $\Omega_{x_0, r} \setminus \Omega_{x_0, r/2}$, for $r > 0$ in a suitable range. Together with the maximum principle, this gives the oscillation of $u$ in $\Omega_{x_0, \delta}$ a decay $\delta^{2\beta}$, for some positive $\varepsilon$-independent $\beta$, where 
$$\delta:= (\varepsilon + |x_0'|^2)^{1/2}.$$ Then we perform a suitable change of variables in $\Omega_{x_0, \delta/4}$, and apply Lemma \ref{gradient_lemma} to obtain the desired estimate on $|\nabla u(x_0)|$.

We fix a $\gamma \in (0,1)$, and let $r_0 >0$ denote a constant depending only on  $n$, $\kappa$, $\gamma$, $R_0$, $\|f\|_{C^{2}}$ and $\|g\|_{C^{2}}$, whose value will be fixed in the proof. We will always consider $0 < \varepsilon \le r_0^2$. First, we require $r_0$ small so that for $|x_0'| < r_0$,
$$10\delta < \delta^{1- \gamma} < R_0/4.$$\\

\begin{lemma}\label{harnack_inequality}
There exists a small $r_0$, depending only on $n, \kappa, \gamma, R_0, \|f\|_{C^{2}}$ and $\|g\|_{C^{2}}$, such that for any $x_0 \in \Omega_{0,r_0}$, $5|x_0'| < r < \delta^{1- \gamma}$, if $u \in H^1(\Omega_{x_0, 2r} \setminus \Omega_{x_0, r/4})$ is a positive solution to the equation
$$
\left\{
\begin{aligned}
-\partial_i (a^{ij}(x) \partial_j u(x)) &=0 \quad \mbox{in }\Omega_{x_0, 2r} \setminus \Omega_{x_0, r/4},\\
a^{ij}(x) \partial_j u(x) \nu_i(x) &= 0 \quad \mbox{on } (\Gamma_+ \cup \Gamma_-) \cap \overline{\Omega_{x_0, 2r} \setminus \Omega_{x_0, r/4}},\\
\end{aligned}
\right.
$$
then,
\begin{equation}\label{harnack}
\sup_{\Omega_{x_0,r} \setminus \Omega_{x_0, r/2}} u \le C \inf_{\Omega_{x_0, r} \setminus \Omega_{x_0, r/2}} u,
\end{equation}
for some constant $C >0$ depending only on $n, \kappa, \lambda, \Lambda, R_0, \|f\|_{C^{2}}$ and $\|g\|_{C^{2}},$ but independent of $r$ and $u$.
\end{lemma}

\begin{proof}
We only need to prove \eqref{harnack} for $|x_0'| > 0$, since the $|x_0'| = 0$ case follows from the result for $|x_0'|>0$ and then sending $|x_0'|$ to $0$. We denote
$$h_r := \varepsilon + f\left(x_0' - \frac{r}{4} \frac{x_0'}{|x_0'|}\right) - g\left(x_0' - \frac{r}{4} \frac{x_0'}{|x_0'|}\right),$$
and perform a change of variables by setting
\begin{equation}\label{x_to_y_1}
\left\{
\begin{aligned}
y' &= x' - x_0' ,\\
y_n &= 2 h_r \left( \frac{x_n - g(x') + \varepsilon/2}{\varepsilon + f(x') - g(x')} - \frac{1}{2} \right),
\end{aligned}\right.
\quad (x',x_n) \in \Omega_{x_0, 2r} \setminus \Omega_{x_0, r/4}.
\end{equation}
This change of variables maps the domain $\Omega_{x_0, 2r} \setminus \Omega_{x_0, r/4}$ to an annular cylinder of height $h_r$, denoted by $Q_{2r, h_r} \setminus Q_{r/4, h_r}$, where
\begin{equation}\label{Q_s_t}
Q_{s,t}:= \{ y = (y',y_n) \in \bR^n ~\big|~  |y'| < s,  |y_n| < t\},
\end{equation}
for $s,t > 0$. We will show that the Jacobian matrix of the change of variables \eqref{x_to_y_1}, denoted by $\partial_x y$, and its inverse matrix $\partial_y x$ satisfy
\begin{equation}\label{transformation_lipschitz}
|(\partial_x y)^{ij}| \le C, \quad |(\partial_y x)^{ij}| \le C, \quad \mbox{for }y \in  Q_{2r, h_r} \setminus Q_{r/4, h_r},
\end{equation}
where $C > 0$ depends only on $n, \kappa, R_0, \|f\|_{C^{2}}$ and $\|g\|_{C^{2}}$.

Let $v(y) = u(x)$, then $v$ satisfies
\begin{equation}\label{equation_v}
\left\{
\begin{aligned}
-\partial_i(b^{ij}(y) \partial_j v(y)) &=0 \quad \mbox{in } Q_{2r, h_r} \setminus Q_{r/4, h_r},\\
b^{nj}(y) \partial_j v(y) &= 0 \quad \mbox{on } \{y_n = -h_r\} \cup \{y_n = h_r\},
\end{aligned}
\right.
\end{equation}
where the matrix $(b^{ij}(y))$ is given by
\begin{equation}\label{b_ij_formula}
(b^{ij}(y)) = \frac{(\partial_x y)(a^{ij})(\partial_x y)^t}{\det (\partial_x y)},
\end{equation}
$(\partial_x y)^t$ is the transpose of $\partial_x y$.

It is easy to see that \eqref{transformation_lipschitz} implies, using $\lambda \le (a^{ij}) \le \Lambda$,
\begin{equation}\label{b_ij_ellpticity}
\frac{\lambda}{C} \le (b^{ij}(y)) \le C\Lambda, \quad \mbox{for }y \in  Q_{2r, h_r} \setminus Q_{r/4, h_r},
\end{equation}
for some constant $C > 0$ depending only on $n, R_0, \kappa, \|f\|_{C^{2}}$ and $\|g\|_{C^{2}}$. 

In the following and throughout this section, we will denote $A \sim B$, if there exists a positive universal constant $C$, which might depend on $n,  \lambda, \Lambda, R_0, \kappa, \|f\|_{C^{2}}$, and $\|g\|_{C^{2}},$ but not depend on $\varepsilon$, such that $C^{-1} B \le A \le C B$.

From \eqref{x_to_y_1}, one can compute that
\begin{align*}
(\partial_x y)^{ii} &= 1, \quad \mbox{for } 1 \le i \le n-1,\\
(\partial_x y)^{nn} &= \frac{2h_r}{\varepsilon + f(x_0'+ y') - g(x_0' + y')},\\
(\partial_x y)^{ni} &= - \frac{2h_r \partial_i g(x_0' + y') + 2y_n [\partial_i f(x_0' + y') - \partial_i g(x_0' + y')]}{\varepsilon + f(x_0' + y')- g(x_0' + y')}, \quad \mbox{for } 1 \le i \le n-1,\\
(\partial_x y)^{ij} &= 0, \quad \mbox{for } 1 \le i \le n-1, j \neq i.
\end{align*}
By \eqref{fg_0} and \eqref{fg_1}, one can see that
$$h_r \sim \varepsilon + \left|x_0' - \frac{r}{4} \frac{x_0'}{|x_0'|}\right|^2.$$
Since $|y_n| \le h_r$, by using \eqref{fg_0} and \eqref{fg_1}, we have that, for $1 \le i \le n-1$,
\begin{align*}
\left|(\partial_x y)^{ni} \right| &\le C\frac{h_r |\partial_i g(x_0' + y')| + h_r [|\partial_i f(x_0' + y')| + |\partial_i g(x_0' + y')|]}{\varepsilon + f(x_0' + y')- g(x_0' + y')}\\
&\le C \frac{h_r}{\varepsilon + f(x_0' + y')- g(x_0' + y')} \left[ |\partial_i f(x_0' + y')| + |\partial_i g(x_0' + y')| \right]\\
&\le C \frac{\varepsilon + \left|x_0' - \frac{r}{4} \frac{x_0'}{|x_0'|}\right|^2}{\varepsilon + |x_0' + y'|^2} |x_0' + y'|, 
\end{align*}
Since $r/4 < |y'| < 2r < 2\delta^{1- \gamma}$ and $|x_0'| < \delta$, we can estimate
\begin{align*}
\left|(\partial_x y)^{ni} \right| \le C|x_0' + y'| \le C(|x_0'| + |y'|) \le C \delta^{1 - \gamma}.
\end{align*} 
Next, we will show that
\begin{equation}\label{partial_x_y_nn}
(\partial_x y)^{nn} \sim 1, \quad \mbox{for }y \in  Q_{2r, h_r} \setminus Q_{r/4, h_r}.
\end{equation}
Indeed, by \eqref{fg_0} and \eqref{fg_1}, we have
$$(\partial_x y)^{nn} = \frac{2h_r}{\varepsilon + f(x_0'+ y') - g(x_0' + y')} \sim \frac{\varepsilon + \left|x_0' - \frac{r}{4} \frac{x_0'}{|x_0'|}\right|^2}{\varepsilon + |x_0' + y'|^2}.$$
Since $|y'| > r/4$, it is easy to see
$$(\partial_x y)^{nn} \le C \frac{\varepsilon + \left|x_0' - \frac{r}{4} \frac{x_0'}{|x_0'|}\right|^2}{\varepsilon + |x_0' + y'|^2} \le C.$$
On the other hand, since $|y'| < 2r$ and $|x_0'| < r/5$, we have
\begin{align*}
\varepsilon + \left|x_0' - \frac{r}{4} \frac{x_0'}{|x_0'|}\right|^2 &\ge \varepsilon + \left( \left| \frac{r}{4} \frac{x_0'}{|x_0'|} \right| - |x_0'| \right)^2\ge \varepsilon + \left( \frac{r}{4} - \frac{r}{5} \right)^2 = \varepsilon + \frac{1}{400}r^2,
\end{align*}
and
\begin{align*}
\varepsilon + |x_0' + y'|^2 &\le \varepsilon + 2|x_0'|^2 + 2|y'|^2 \le \varepsilon + \frac{2}{25}r^2 + 8r^2 < \varepsilon + 9r^2.
\end{align*}
Therefore,
$$(\partial_x y)^{nn} \ge \frac{1}{C} \frac{\varepsilon + \left|x_0' - \frac{r}{4} \frac{x_0'}{|x_0'|}\right|^2}{\varepsilon + |x_0' + y'|^2} \ge \frac{1}{C} \frac{\varepsilon + r^2/400}{\varepsilon + 9 r^2} \ge \frac{1}{C},$$
and \eqref{partial_x_y_nn} is verified.

We have shown $(\partial_x y)^{ii} \sim 1$, for all $i = 1, \cdots, n$, and $|(\partial_x y)^{ij}| \le C \delta^{1-\gamma}$, for $i \neq j$. We further require $r_0$ to be small enough so that off-diagonal entries of $\partial_x y$ are small. Therefore \eqref{transformation_lipschitz} follows.  As mentioned earlier, \eqref{b_ij_ellpticity} follows from \eqref{transformation_lipschitz}.

Now we define, for any integer $l$,
$$A_l:= \left\{y \in \bR^n ~\big|~  \frac{r}{4} < |y'| < 2r,~ (l-1) h_r < z_n < (l+1) h_r \right\}.$$
 Note that $A_0 = Q_{2r, h_r} \setminus Q_{r/4, h_r}.$ For any $l \in \bZ$, we define a new function $\tilde{v}$ by
$$\tilde{v}(y) := v\left(y', (-1)^l\left(y_n - 2l h_r\right)\right), \quad \forall y \in A_l.$$
We also define the corresponding coefficients, for $k = 1,2, \cdots, n-1$,
$$\tilde{b}^{nk}(y)=\tilde{b}^{kn}(y) := (-1)^lb^{nk}\left(y', (-1)^l\left(y_n - 2l h_r\right)\right),  \quad \forall y \in A_l,$$
and for other indices,
$$\tilde{b}^{ij}(y) := b^{ij}\left(y', (-1)^l\left(y_n - 2l h_r\right)\right), \quad \forall y \in A_l.$$
Therefore, $\tilde{v}(y)$ and $\tilde{b}^{ij}(y)$ are defined in the infinite cylinder shell $Q_{2r, \infty} \setminus Q_{r/4, \infty}$. By \eqref{equation_v}, $\tilde{v} \in H^1(Q_{2r, \infty} \setminus Q_{r/4, \infty})$ satisfies
$$-\partial_i (\tilde{b}^{ij}(y) \partial_j \tilde{v}(y)) = 0 \quad \mbox{in }Q_{2r, \infty} \setminus Q_{r/4, \infty}.$$
Note that for any $l \in \bZ$ and $y \in A_l$, $\tilde{b}(y) = (\tilde{b}^{ij}(y))$ is orthogonally conjugated to $b\left(y', (-1)^l\left(y_n - 2l h_r\right)\right)$. Hence, by \eqref{b_ij_ellpticity}, we have
$$\frac{\lambda}{C} \le \tilde{b}(y) \le C\Lambda, \quad \mbox{for }y \in  Q_{2r, \infty} \setminus Q_{r/4, \infty}.$$

We restrict the domain to be $Q_{2r, r} \setminus Q_{r/4, r}$, and make the change of variables $z = y/r$. Set $\bar{v}(z) = \tilde{v}(y), \bar{b}^{ij}(z) = \tilde{b}^{ij}(y)$, we have
$$-\partial_i (\bar{b}^{ij}(z) \partial_j \bar{v}(z)) = 0 \quad \mbox{in }Q_{2, 1} \setminus Q_{1/4, 1},$$
and
$$\frac{\lambda}{C} \le \bar{b}(z) \le C\Lambda, \quad \mbox{for }z \in  Q_{2, 1} \setminus Q_{1/4, 1}.$$
Then by the Harnack inequality for uniformly elliptic equations of divergence form, see e.g. \cite[Theorem 8.20]{GT}, there exists a constant $C$ depending only on $n, \kappa, \lambda, \Lambda, R_0, \|f\|_{C^{2}}$ and $\|g\|_{C^{2}},$ such that
$$\sup_{Q_{1,1/2} \setminus Q_{1/2,1/2}} \bar{v} \le C \inf_{Q_{1,1/2} \setminus Q_{1/2,1/2}} \bar{v}.$$
In particular, we have
$$\sup_{Q_{1,h_r/r} \setminus Q_{1/2,h_r/r}} \bar{v} \le C \inf_{Q_{1,h_r/r} \setminus Q_{1/2,h_r/r}} \bar{v},$$
which is \eqref{harnack} after reversing the change of variables.\\
\end{proof}

\begin{remark}
If dimension $n = 2$, Lemma \ref{harnack_inequality} fails since $Q_{2, 1} \setminus Q_{1/4, 1} \subset \bR^{2}$ is the union of two disjoint rectangular domains, and the Harnack inequality cannot be applied on it. In fact, in our proof of Theorem \ref{main_thm}, Lemma \ref{harnack_inequality} is the only ingredient where dimension $n \ge 3$ is used. As mentioned above, the conclusion of Theorem \ref{main_thm} does not hold in dimension $n = 2$.\\
\end{remark}

For any domain $A \subset \widetilde{\Omega}$, we denote the oscillation of $u$ in A by $\mbox{osc}_A u := \sup_{A} u - \inf_{A} u$. Using Lemma \ref{harnack_inequality}, we obtain a decay of $\mbox{osc}_{\Omega_{x_0, \delta}}u$ in $\delta$ as follows.\\

\begin{lemma}\label{osc_u_decay_lemma}
Let $u$ be a solution of \eqref{main_problem_narrow}. For any $x_0 \in \Omega_{0, r_0}$, where $r_0$ is as in Lemma \ref{harnack_inequality}, there exist positive constants $\sigma$ and $C$, depending only on $ n, R_0, \kappa, \|f\|_{C^{2}}$ and $\|g\|_{C^{2}},$ such that
\begin{equation}\label{osc_u}
\mbox{osc}_{\Omega_{x_0, \delta}} u \le C \| u\|_{L^\infty(\Omega_{x_0, \delta^{1 - \gamma}})} \delta^{\gamma \sigma}.
\end{equation}
\end{lemma}

\begin{proof}
For simplicity, we drop the $x_0$ subscript and denote $\Omega_r = \Omega_{x_0,r}$ in this proof. Let $5|x_0'| < r < \delta^{1- \gamma}$ and $u_1 = \sup_{\Omega_{2r}}u - u, u_2 = u - \inf_{\Omega_{2r}} u.$ By Lemma \ref{harnack_inequality}, we have 
\begin{align*}
\sup_{\Omega_r \setminus \Omega_{r/2}} u_1 &\le C_1 \inf_{\Omega_r \setminus \Omega_{r/2}} u_1, \\
\sup_{\Omega_r \setminus \Omega_{r/2}} u_2 &\le C_1 \inf_{\Omega_r \setminus \Omega_{r/2}} u_2,
\end{align*}
where $C_1 > 1$ is a constant independent of $r$. Since both $u_1$ and $u_2$ satisfy equation \eqref{main_problem_narrow}, by the maximum principle,
\begin{align*}
\sup_{\Omega_r \setminus \Omega_{r/2}} u_i = \sup_{\Omega_r} u_i, \quad \inf_{\Omega_r \setminus \Omega_{r/2}} u_i = \inf_{\Omega_r} u_i,
\end{align*}
for $i = 1,2$. Therefore,
\begin{align*}
\sup_{\Omega_r} u_1 &\le C_1 \inf_{\Omega_r} u_1, \\
\sup_{\Omega_r} u_2 &\le C_1 \inf_{\Omega_r} u_2.
\end{align*}
Adding up the above two inequalities, we have
$$\mbox{osc}_{\Omega_r} u \le \left( \frac{C_1 - 1}{C_1 + 1} \right)\mbox{osc}_{\Omega_{2r}} u.$$
Now we take $\sigma > 0$ such that $2^{-\sigma} = \frac{C_1 - 1}{C_1 + 1}$, then
\begin{equation}\label{recurrence}
\mbox{osc}_{\Omega_r} u  \le 2^{-\sigma} \mbox{osc}_{\Omega_{2r}} u.
\end{equation}
We start with $r = r_0 = \delta^{1- \gamma}/2$, and set $r_{i+1} = r_i/2$.
Keep iterating \eqref{recurrence} $k+1$ times, where $k$ satisfies $5\delta \le r_k < 10 \delta$, we will have
$$\mbox{osc}_{\Omega_{\delta}} u \le \mbox{osc}_{\Omega_{r_k}} u \le  2^{-(k+1)\sigma} \mbox{osc}_{\Omega_{2r_0}} u \le 2^{1-(k+1)\sigma} \|u\|_{L^\infty (\Omega_{\delta^{1-\gamma}})}.$$
Since $10\delta > r^{k} = 2^{-k}r_0 = 2^{-(k+1)}\delta^{1 - \gamma},$ we have $ 2^{-(k+1)} < 10 \delta^\gamma$, and hence \eqref{osc_u} follows immediately.\\
\end{proof}

\begin{proof}[Proof of Theorem \ref{main_thm}]
Let $u \in H^1(\Omega_{0,R_0})$ be a solution of \eqref{main_problem_narrow}. For $x_0 \in \Omega_{0,r_0}$, we have, using Lemma \ref{osc_u_decay_lemma},
\begin{equation}\label{u-u_0}
\| u - u_0\|_{L^\infty(\Omega_{x_0, \delta})} \le C \| u\|_{L^\infty(\Omega_{x_0, \delta^{1 - \gamma}})} \delta^{\gamma \sigma},
\end{equation}
for some constant $u_0$. We denote $v: = u - u_0$, and $v$ satisfies the same equation \eqref{main_problem_narrow}. We work on the domain $\Omega_{x_0,  \delta/4}$, and perform a change of variables by setting
\begin{equation}\label{x_to_y}
\begin{cases}
y' = \delta^{-1} (x'- x_0'),\\
y_n = \delta^{-1} x_n.
\end{cases}
\end{equation}
The domain $\Omega_{x_0, \delta/4}$ becomes
\begin{align*}
\left\{y\in \bR^n ~\big|~ |y'| \le \frac{1}{4}, \delta^{-1} \left( -\frac{1}{2}\varepsilon+ g(x_0' + \delta y')\right)< y_n < \delta^{-1} \left( \frac{1}{2}\varepsilon+ f(x_0'+ \delta y')\right) \right\}.
\end{align*}
We make a change of variables again by
\begin{equation}\label{y_to_z}
\begin{cases}
z' = 4y' ,\\
z_n = 2\delta \left( \frac{\delta y_n - g(x_0' + \delta y') + \varepsilon/2}{\varepsilon + f(x_0' + \delta y') - g(x_0' + \delta y')} - \frac{1}{2} \right).
\end{cases}
\end{equation}
Now the domain in $z$-variables becomes a thin plate $Q_{1, \delta}$, where $Q_{s,t}$ is defined as in \eqref{Q_s_t}. Let $w(z) = v(x)$, then $w$ satisfies
\begin{equation}\label{equation_w}
\left\{
\begin{aligned}
-\partial_i(b^{ij}(z) \partial_j w(z)) &=0 \quad \mbox{in } Q_{1, \delta},\\
b^{nj}(z) \partial_j w(z) &= 0 \quad \mbox{on } \{z_n = -\delta\} \cup \{z_n = \delta\},
\end{aligned}
\right.
\end{equation}
where the matrix $b(z) = (b^{ij}(z))$ is given by
\begin{equation}\label{b_ij_formula_2}
(b^{ij}(z)) = \frac{(\partial_y z)(a^{ij})(\partial_y z)^t}{\det (\partial_y z)}.
\end{equation}

Similar to the proof of Lemma \ref{harnack_inequality}, we will show that the Jacobian matrix of the change of variables \eqref{y_to_z}, denoted by $\partial_y z$, and its inverse matrix $\partial_z y$ satisfy
\begin{equation}\label{transformation_lipschitz_2}
|(\partial_y z)^{ij}| \le C, \quad |(\partial_z y)^{ij}| \le C, \quad \mbox{for }z \in  Q_{1, \delta},
\end{equation}
where $C > 0$ depends only on $n, \kappa, R_0, \|f\|_{C^{2}}$ and $\|g\|_{C^{2}}$. This leads to
\begin{equation}\label{b_ij_ellpticity_2}
\frac{\lambda}{C} \le b(z) \le C\Lambda, \quad \mbox{for }z \in  Q_{1, \delta}.
\end{equation}
From \eqref{y_to_z}, one can compute that
\begin{align*}
(\partial_y z)^{ii} &= 4, \quad \mbox{for } 1 \le i \le n-1,\\
(\partial_y z)^{nn} &= \frac{2\delta^2}{\varepsilon + f(x_0' + \delta z'/4) - g(x_0' + \delta z'/4)},\\
(\partial_y z)^{ni} &= - \frac{2 \delta^2 \partial_i g(x_0' + \delta z'/4) + 2z_n \delta [\partial_i f(x_0' + \delta z'/4) - \partial_i g(x_0' + \delta z'/4)]}{\varepsilon + f(x_0' + \delta z'/4)- g(x_0' + \delta z'/4)}\\
&~\quad \mbox{for } 1 \le i \le n-1,\\
(\partial_y z)^{ij} &= 0, \quad \mbox{for } 1 \le i \le n-1, j \neq i.
\end{align*}
First we will show that
\begin{equation}\label{partial_y_z_nn}
(\partial_y z)^{nn} \sim 1, \quad \mbox{for }z \in Q_{1, \delta}.
\end{equation}
Since $|z'| < 1$ and $|x_0'| < \delta$, it is easy to see that
$$(\partial_y z)^{nn} \sim \frac{\delta^2}{\varepsilon + |x_0' + \delta z'/4|^2} \ge \frac{\delta^2}{\varepsilon + C \delta^2} \ge \frac{1}{C}, \quad \mbox{for }z \in Q_{1, \delta},$$
due to \eqref{fg_0} and \eqref{fg_1}. On the other hand,
\begin{align*}
(\partial_y z)^{nn} &\sim \frac{\delta^2}{\varepsilon + |x_0' + \delta z'/4|^2}\\
&= \frac{\delta^2}{\varepsilon + |x_0'|^2 + (1/4)^2 \delta^2 |z'|^2 + \delta x_0' \cdot z'/2}\\
&\le \frac{\delta^2}{\delta^2 + (1/4)^2|z'|^2\delta^2 - |z'||x_0'|\delta/2}\\
&\le \frac{\delta^2}{(1 + (1/4)^2|z'|^2 - 1/2) \delta^2} \le C, \quad \mbox{for }z \in Q_{1, \delta}.
\end{align*}
Therefore, \eqref{partial_y_z_nn} is verified.

Since $|z_n| < \delta$, $|z'| < 1$ and $|x_0'| < \delta$, by \eqref{fg_0} and \eqref{fg_1}, for $1 \le i \le n-1$,
\begin{align*}
|(\partial_y z)^{ni}| &\le \frac{2 \delta^2 |\partial_i g(x_0' + \delta z'/4)| + 2 \delta^2 [|\partial_i f(x_0' + \delta z'/4)| + |\partial_i g(x_0' + \delta z'/4)|]}{\varepsilon + f(x_0' + \delta z'/4)- g(x_0' + \delta z'/4)}\\
&\le \frac{C\delta^2}{\varepsilon + f(x_0' + \delta z'/4) - g(x_0' + \delta z'/4)}[|\partial_i f(x_0' + \delta z'/4)| + |\partial_i g(x_0' + \delta z'/4)|]\\
&\le C\frac{\delta^2}{\varepsilon + |x_0' + \delta z'/4|^2} |x_0' + \delta z'/4|\\
&\le C (|x_0'| + \delta|z'|) \le C\delta,
\end{align*}
where in the last line, we have used the same arguments in showing $(\partial_y z)^{nn} \le C$ earlier.

We have shown $(\partial_y z)^{ii} \sim 1$, for all $i = 1, \cdots, n$, and $|(\partial_y z)^{ij}| \le C \delta$, for $i \neq j$. We further require $r_0$ to be small enough so that off-diagonal entries are small. Therefore \eqref{transformation_lipschitz_2} follows.  As mentioned earlier, \eqref{b_ij_ellpticity_2} follows from \eqref{transformation_lipschitz_2}.

Next, we will show
\begin{equation}\label{b_ij_holder}
\|b \|_{C^\alpha(\overline{Q}_{1,\delta})} \le C,
\end{equation}
for some $C > 0$ depending only on $n, \kappa, R_0, \|a\|_{C^\alpha}, \|f\|_{C^{2}}$ and $\|g\|_{C^{2}}$, by showing
\begin{equation}\label{partial_y_z_lipschitz}
|\nabla_z (\partial_y z)^{ij}(z)| \le C, \quad \left| \nabla_z \frac{1}{\det(\partial_y z)} \right| \le C,  \quad \mbox{for }z \in Q_{1, \delta}.
\end{equation}
Then \eqref{b_ij_holder} follows from \eqref{partial_y_z_lipschitz}, \eqref{b_ij_formula_2}, and $\|a\|_{C^\alpha} \le C$.

By a straightforward computation, we have, for any $i = 1, \cdots, n-1$,
\begin{align*}
\left| \partial_{z_i} \frac{1}{\det(\partial_y z)} \right| &= \left| \partial_{z_i} \left( \frac{\varepsilon + f(x_0' + \delta z'/4) - g(x_0' + \delta z'/4)}{2 \cdot 4^{n-1}\delta^2} \right) \right|\\
&= \left| \frac{\delta[\partial_i f(x_0' + \delta z'/4) - \partial_i g(x_0' + \delta z'/4)]}{2 \cdot 4^{n-1}\delta^2}  \right|\\
&\le \frac{C}{\delta}[|\partial_i f(x_0' + \delta z'/4)| + |\partial_i g(x_0' + \delta z'/4)|]\\
&\le \frac{C}{\delta}|x_0' + \delta z'/4| \le C, \quad \mbox{for }z \in Q_{1, \delta},
\end{align*}
where in the last inequality, \eqref{fg_0} and \eqref{fg_1} have been used. For any $i = 1, \cdots, n-1$,
\begin{align*}
|\partial_{z_i} (\partial_y z)^{nn}| &=  \left| \frac{2\delta^3 [\partial_i f(x_0' + \delta z'/4) - \partial_i g(x_0' + \delta z'/4)]}{(\varepsilon + f(x_0' + \delta z'/4) - g(x_0' + \delta z'/4))^2} \right|\\
&\le \frac{C\delta^3}{(\varepsilon + |x_0' + \delta z'/4|^2)^2}|x_0' + \delta z'/4|\\
&\le \frac{C\delta^3}{\delta^4}(|x_0'| + |\delta z'|) \le C, \quad \mbox{for }z \in Q_{1, \delta},
\end{align*}
where in the last line, we have used the same arguments in showing $(\partial_y z)^{nn} \le C$ earlier. Similar computations apply to $\partial_{z_i} (\partial_y z)^{ni}$, for $i = 1, \cdots, n-1$, and we have
$$|\partial_{z_i} (\partial_y z)^{ni}| \le C, \quad \mbox{for }z \in Q_{1, \delta}.$$
Finally, we compute, for $i = 1, \cdots, n-1$,
\begin{align*}
|\partial_{z_n} (\partial_y z)^{ni}| &= \left| \frac{ 2 \delta [\partial_i f(x_0' + \delta z'/4) - \partial_i g(x_0' + \delta z'/4)]}{\varepsilon + f(x_0' + \delta z'/4)- g(x_0' + \delta z'/4)}  \right|\\
&\le \frac{C\delta|x_0' + \delta z'/4|}{\varepsilon + |x_0' + \delta z'/4|^2} \le C, \quad \mbox{for }z \in Q_{1, \delta}.
\end{align*}
Therefore, \eqref{partial_y_z_lipschitz} is verified, and hence \eqref{b_ij_holder} follows as mentioned above.

Now we define
$$S_l:= \left\{z \in \bR^n ~\big|~  |z'| < 1,~ (l-1) \delta < z_n < (l+1) \delta \right\}$$
for any integer $l$, and
$$S: = \left\{z \in \bR^n ~\big|~  |z'| < 1,~ |z_n| < 1\right\}.$$
Note that $Q_{1, \delta} = S_0$. As in the proof of Lemma \ref{harnack_inequality}, we define, for any $l \in \bZ$, a new function $\tilde{w}$ by setting
$$\tilde{w}(z) := w\left(z', (-1)^l\left(z_n - 2l \delta\right)\right), \quad \forall z \in S_l.$$
We also define the corresponding coefficients, for $k = 1,2, \cdots, n-1$,
$$\tilde{b}^{nk}(z)=\tilde{b}^{kn}(z) := (-1)^lb^{nk}\left(z', (-1)^l\left(z_n - 2l \delta\right)\right),  \quad \forall z \in S_l,$$
and for other indices,
$$\tilde{b}^{ij}(z) := b^{ij}\left(z', (-1)^l\left(z_n - 2l \delta\right)\right), \quad \forall y \in S_l.$$
Then $\tilde{w}$ and $\tilde{b}^{ij}$ are defined in the infinite cylinder $Q_{1, \infty}$. By \eqref{equation_w}, $\tilde{w}$ satisfies the equation
$$-\partial_i (\tilde{b}^{ij} \partial_j \tilde{w}) = 0, \quad \mbox{in }Q_{1, \infty}.$$
Note that for any $l \in \bZ$, $\tilde{b}(z)$ is orthogonally conjugated to $b\left(z', (-1)^l\left(z_n - 2l \delta\right)\right),$ for $z \in S_l$. Hence, by \eqref{b_ij_ellpticity_2}, we have
\begin{equation*}
\frac{\lambda}{C} \le \tilde{b}(z) \le C\Lambda, \quad \mbox{for }z \in Q_{1,\infty},
\end{equation*}
and, by \eqref{b_ij_holder},
\begin{equation*}
\|\tilde{b} \|_{C^\alpha(\overline{S}_{l})} \le C, \quad \forall l \in \bZ.
\end{equation*}
Apply Lemma \ref{gradient_lemma} on $S$ with $N = 1$, we have
$$\| \nabla \tilde{w} \|_{L^\infty(\frac{1}{2}S)} \le C \| \tilde{w} \|_{L^2(S)}.$$
It follows that
$$\| \nabla w \|_{L^\infty(Q_{1/2, \delta})} \le \frac{C}{\delta} \| w \|_{L^2(Q_{1, \delta})} \le C\|w\|_{L^\infty(Q_{1, \delta})},$$
for some positive constant $C$, depending only on $n, \alpha, R_0, \kappa, \lambda, \Lambda, \|a\|_{C^\alpha}, \|f\|_{C^{2}}$ and $\|g\|_{C^{2}}$.

Since $\|(\partial_z y)\|_{L^\infty(Q_{1,\delta})} \le C$ by \eqref{transformation_lipschitz_2}, where $C$ depends only on $R_0, \kappa,  \|f\|_{C^{2}}$ and $\|g\|_{C^{2}}$, and in particular, is independent of $\varepsilon$ and $\delta$. Reversing the change of variables \eqref{y_to_z} and \eqref{x_to_y}, we have
$$\delta \| \nabla v\|_{L^\infty(\Omega_{x_0, \delta/8})} \le C \|v\|_{L^\infty(\Omega_{x_0, \delta/4})} \le C \| u\|_{L^\infty(\Omega_{x_0, \delta^{1 - \gamma}})} \delta^{\gamma \sigma}$$
by \eqref{u-u_0}. In particular, this implies
$$|\nabla u(x_0)| \le C \| u\|_{L^\infty(\Omega_{x_0, \delta^{1 - \gamma}})} \delta^{-1 + \gamma \sigma},$$
and it concludes the proof of Theorem \ref{main_thm} after taking $\beta = \gamma \sigma/2$.\\
\end{proof}

\section{Gradient estimates of elliptic systems}

A natural question is whether the estimate in Theorem \ref{main_thm} can be extended to elliptic systems of divergence form. We tend to believe that the answer to this question is affirmative, and plan to pursue this in a subsequent paper. Following closely the proof of \eqref{insulated_upper_bound} in \cite{BLY2}, we give a preliminary gradient estimate of elliptic systems in this section.

We consider the vector-valued function $u = (u_1, \cdots, u_N)$, and for $1 \le \alpha, \beta \le n, 1 \le i,j \le N$, let $A^{\alpha \beta}_{ij}(x)$ be a function such that
$$\| A^{\alpha \beta}_{ij} \|_{L^\infty(\Omega_{0,R_0})} \le \Lambda,$$
$$\int_{\Omega_{0,R_0}} A^{\alpha \beta}_{ij}(x) \partial_\alpha \varphi_i(x) \partial_\beta \varphi_j(x) \ge \lambda \int_{\Omega_{0,R_0}} |\nabla \varphi|^2, \quad \forall \varphi \in H_0^1(\Omega_{0,R_0}; \bR^N),$$
for some $\lambda, \Lambda > 0$, where $\Omega_{0,R_0}$ is defined as in \eqref{domain_def_Omega}.
We assume $A^{\alpha \beta}_{ij}(x) \in C^\mu(\Omega_{0,R_0})$ for some $\mu \in(0,1)$, and consider the system
\begin{equation}\label{system}
\left\{
\begin{aligned}
-\partial_\alpha \left(A^{\alpha \beta}_{ij}(x) \partial_\beta u_j(x)\right) &=0 \quad \mbox{in }\Omega_{0,R_0},\\
A^{\alpha \beta}_{ij}(x) \partial_\beta u_j(x) \nu_\alpha(x) &= 0 \quad \mbox{on } \Gamma_+ \cup \Gamma_-,\\
\end{aligned}
\right.
\end{equation}
for $i = 1,\cdots, N$, where $\Gamma_+, \Gamma_-$ are defined as in \eqref{domain_def_Omega}, $\nu = (\nu_1, \cdots, \nu_n)$ denotes the unit normal vector on $\Gamma_+$ and $\Gamma_-$, pointing upward and downward respectively. We have the following gradient estimate by essentially following the proof of Theorem 1.2 in \cite{BLY2}.\\

\begin{theorem}\label{system_thm}
Let $u\in H_0^1(\Omega_{0,R_0}; \bR^N)$ be a solution to \eqref{system} in dimension $n \ge 2$, with the coefficient $A^{\alpha \beta}_{ij}$ defined as above. There exist positive constants $r_0$ and $C$ depending only on $n$, $\lambda$, $\Lambda$, $R_0$, $\kappa$, $\mu$, $\|A\|_{C^\mu(\Omega_{0,R_0})}$, $\|f\|_{C^{2}(\{|x'| \le R_0\})}$ and $\|g\|_{C^{2}(\{|x'| \le R_0\})},$ such that
\begin{equation}\label{main_result_system}
|\nabla u (x_0)| \le C  \| u\|_{L^\infty(\Omega_{0,R_0})} \left(\varepsilon + |x_0'|^2 \right) ^{-1/2},
\end{equation}
for all $\varepsilon \in (0,1)$, $x_0 \in \Omega_{0 , r_0}$.\\
\end{theorem}

\begin{remark}
The elliptic systems we have considered include the linear systems of elasticity:  $n = N$, and the coefficients $A^{\alpha \beta}_{ij}$ satisfy
$$A^{\alpha \beta}_{ij} = A^{\beta \alpha}_{ji} = A^{i \beta}_{\alpha j},$$
and for all $n \times n$ symmetric matrices $\{\xi_{\alpha}^i\}$,
$$\lambda |\xi|^2 \le A^{\alpha \beta}_{ij}\xi_{\alpha}^i \xi_{\beta}^j \le \Lambda|\xi|^2.$$\\
\end{remark}

\begin{proof}[Proof of Theorem \ref{system_thm}]
Let $u\in H^1(\Omega_{0,R_0}; \bR^N)$ be a solution to \eqref{system}. We perform the changes of variables \eqref{x_to_y} and \eqref{y_to_z}. For any $1 \le i,j \le N$, we define
$$B_{ij}^{\alpha \beta}(z) = \frac{(\partial_y z)(A^{\alpha \beta}_{ij})(\partial_y z)^t}{\det (\partial_y z)},$$
and let $v(z) = u(x)$. Then $v$ satisfies
\begin{equation*}
\left\{
\begin{aligned}
-\partial_\alpha \left(B^{\alpha \beta}_{ij}(z) \partial_\beta v_j(z)\right) &=0 \quad \mbox{in } Q_{1, \delta},\\
B^{n \beta}_{ij}(z) \partial_\beta v_j(z) &= 0 \quad \mbox{on } \{z_n = -\delta\} \cup \{z_n = \delta\},
\end{aligned}
\right.
\end{equation*}
for $i = 1,\cdots, N$, where $Q_{s,t}$ is defined as in \eqref{Q_s_t}. As in the proof of Theorem \ref{main_thm}, we can show that
$$\| B^{\alpha \beta}_{ij} \|_{L^\infty(Q_{1,\delta})} \le C\Lambda, \quad \| B^{\alpha \beta}_{ij} \|_{C^\mu(\bar{Q}_{1,\delta})} \le C,$$
$$\int_{Q_{1,\delta}} B^{\alpha \beta}_{ij}(z) \partial_\alpha \varphi_i(z) \partial_\beta \varphi_j(z) \ge \frac{\lambda}{C} \int_{Q_{1,\delta}} |\nabla \varphi|^2, \quad \forall \varphi \in H_0^1(Q_{1,\delta}; \bR^N),$$
where $C$ is a positive constant that depends only on $n$, $N$, $\mu$, $R_0$, $\kappa$, $\lambda$, $\Lambda$, $\|A\|_{C^\mu}$, $\|f\|_{C^{2}}$ and $\|g\|_{C^{2}}$. Then we argue as in the proof of Theorem \ref{main_thm} to obtain
$$|\nabla v(0)| \le C\|v\|_{L^\infty(Q_{1, \delta})},$$
which is \eqref{main_result_system} after reversing the changes of variables \eqref{x_to_y} and \eqref{y_to_z}.\\
\end{proof}

\end{document}